\documentclass[12pt]{amsart}
\usepackage{amssymb,latexsym}

\begin{document}

\newcommand{\lab}[1]{\label{#1}\marginpar{\footnotesize #1}}

\newcommand{\bi}{b}
\newcommand{\di}{d}

\newtheorem{thm}{Theorem}
\newtheorem{pro}[thm]{Proposition}
\newtheorem{cor}[thm]{Corollary}
\newtheorem{lem}[thm]{Lemma}
\newtheorem{prob}[thm]{Problem}
\newtheorem{obs}[thm]{Observation}
\newtheorem{fact}[thm]{Fact}
\newtheorem{add}[thm]{Addendum}
\newtheorem{ex}[thm]{Example}
\newcommand{\ov}[1]{\overline{#1}}

\newcommand{\mb}[1]{\mathbb{#1}}
\newcommand{\mc}[1]{\mathcal{#1}}
\newcommand{\co}[1]{}

\newcommand{\im}{{\sf im}}
\newcommand{\vep}{\varepsilon}

%%%%%%%%%%%%%%%%%%%%%%%%%%%%%%%%%%%%%%%%%%%%%%%%%%%%%%%%%%%%%%%%%%%%%%%%%%%%%% 
\title[Lattices of equivalenceses]{Some remarks on  lattices of equivalences }

\author[C. Herrmann]{Christian Herrmann}
\address{Technische Universit\"{a}t Darmstadt FB4\\Schlo{\ss}gartenstr. 7, 64289 Darmstadt, Germany}
\email{herrmann@mathematik.tu-darmstadt.de}

\maketitle

\section{ Closure properties of classes of representale lattices}

Relying on  Mal'cev's method of axiomatic correspondences (cf. \cite{hp}), 
we consider classes $\mc{A}$ of multi-sorted  structures  $A$
with two   designated  sorts $S$ and $L$
where   $L$ is endowed with a lattice structure 
and where $\rho \subseteq S\times S \times L$ is a relation, 
the only relation or operation
relating sort $L$ with other sorts.
Moreover, the  following are required
\begin{enumerate}
\item  $\vep(a):= \{(x,y) \mid x,y \in S,  \rho(x,y,a)\}$ is
an equivalence relation on $S$ for each $a \in L$.
\item  $a \mapsto \vep(a)$ is an injective map of $L$ 
into the set $\Pi(S)$  of all equivalence relations on $S$. 
\item $\vep(a \cdot b)= \vep(a) \cap \vep(b)$ for all $a,b \in L$.
\item
 There is $n\in \mb{N}$ (uniformly for all $A\in \mc{A}$) such that,
for all $a,b \in L$, $\vep(a+b)$ 
is the $n$-fold relational product of $\vep(a)$ and $\vep(b)$.
\end{enumerate}

For $A \in \mc{A}$ let $A^-$ denote the structure obtained by
removing $L$ and $\rho$; let $A|L$ denote
sort $L$ of $A$ with  its lattice structure..
 Let $\mc{A}^-$ and $\mc{A}|L$ denote 
the class of all $A^-$ resp. $A|L$ where $A \in \mc{A}$.

\begin{fact}
$\mc{A}|L$ is closed under formation of sublattices.
\end{fact}
\begin{proof}
If $A \in  \mc{A}$ and if $L'$ is  a sublattice of
sort $L$ of $A$,   then $A'\in \mc{A}$ where $L$ is replaced
by $L'$ and  $\rho$ by the restriction to $S\times S\times L'$.
\end{proof}

\begin{fact}
If $\mc{A}^-$ is an axiomatic class  then so is $\mc{A}$. 
 If $\mc{A}^-$ is closed under ultraproducts 
then so are $\mc{A}$ and $\mc{A}|L$.
\end{fact}
\begin{proof}
In addition to the axioms of $\mc{A}^-$ 
 one has to require that $L$ is a lattice 
and that $\vep$ defines an embedding into 
a lattice of $n$-permuting equivalences on $S$.
This can be stated via first order formulas
in the language of $\mc{A}$. Finally observe that the ultraproduct
of a multi-sorted structure is constructed sort-wise.

\end{proof}

\begin{fact} If $\mc{A}^-$ is closed under direct products then
so are $\mc{A}$ and $\mc{A}|L$.
\end{fact}

\begin{proof} First, obeserve that direct products
of multi-sorted structures are constructed sort-wise.
Now, let $A=\prod_{i\in I}A_i$, in particular 
\[ \rho((x_i| i \in I),(y_i| i\in I),(a_i \mid i\in I))
 \mbox{ iff } \rho_i(x_i,y_i,a_i) \mbox{ for all }  i\in I   \]
that is
\[  \vep((a_i|i\in I))= (\vep_i(a_i)\mid i \in I). \]
It follows for $a=(a_i|i \in I)$ and $b=(b_i| i\in I)$ in $L$: 
\[ \vep(a \cdot b)= (\vep_i(a_i\cdot b_i) 
| \in I) =(\vep_i(a_i) \cap \vep_i(b_i)|i \in I)= \vep(a)  \cap \vep(b)\]
\[ \vep(a + b)= (\vep_i(a_i+ b_i) 
| \in I) =(\vep_i(a_i) \bowtie_n \vep_i(b_i)|i \in I)= \vep(a)  \bowtie_n \vep(b)\]
where
$\alpha \bowtie_n  \beta= \alpha\circ \beta \circ \alpha \circ \ldots $       
stands for the $n$-termed relational product of $\alpha$ and $\beta$.

\end{proof}

\begin{cor}
 (Folklore) If $\mc{A}^-$ is closed under ultraproducts and direct 
products then $\mc{A}|L$ is a quasivariety.
\end{cor}

In particular, each of the following classes of lattices
is a quasivariety. 
\begin{enumerate}
\item The class of all lattices embeddable into  lattices
 of permuting equivalences. 
\item The class of all lattices embeddable into
congruence lattices of algebras in a congruece permutable
quasivariety.
\end{enumerate}
Observe that embedding into lattices of normal subgroups, ideals, resp.
submodules ($1$- or $2$-sorted) subsumes under (2).

\begin{prob}
Is any of the above mentioned classes 
 closed   under homomorphic images?
\end{prob}

\begin{prob}
If $L$ embeds into the subspace lattice of some finite vector space $V_F$,
does every homomorphic image of $L$ embed into some lattice
of permuting equivalences?
\end{prob}
A naive appoasch would be to construct an embedding  
into an ideal of $L(V)$ choosing atomic supplements for 
join irreducibles. This fails in view of an example 
provided by the theory of linear representations of posets
(cf. \cite{contr}).

For the case of $n$-permutable representations the answer
and a finite axiomatization  has been 
given by J\`onsson: $n=3$ all modular lattices, $n>3$ all lattices.

\section{Examples}
Examples of modular (and even Argueesian)  lattices not having a particular kind of represebtation
have been given for the following cases.
\begin{enumerate}
\item No vector space representation (Dilworth and Hall \cite{dil})
\item No representation in normal subgroup lattices (J\`onsson \cite{jon2})
\item No representation in congruence lattices of algebras
from a congruence modular variety (Freese, Herrmann, Huhn \cite{f0} 
\item No representation by permuting equivalences (Haiman \cite{hai1,hai2})
\end{enumerate}
The common feature of (1)-(3) is that one combines subspace  lattices $L_i$, $i=1,2$,
of height $3$ and          characteristic $p_1 \# p_2$  a modular lattice
$L$ so that the prime quotients of the  $L_i$ are projective in $L$
+ we write $p_1 \# p_2$ if $p_1,p_2$ are distinct primes or $p_1=p_2=0$.
A particular construction is as follows. Let $S_1=M_3$
the $5$-element modular lattice.
Now  construct a sequence of  simple finite
planar modular lattices $S_n$ (called \emph{snakes}) of length $n+1$
and width $4$
 such that $S_{n-1}=[0,c]_{S_n}$, This is done 
by gluing a prime  quotient $a/0$ of $M_3$ with $1_{S_n}/c$
where $c$ is the unique coatom of $S_n$
 which is in $S_{n-1}$. Obsere that in $S_n$ one has 
$a/0$ projective to $1/c$ for any atoms $a$ and coatoms $c$ but
not in less than $n$ steps.  Now, with $L_i$ from above,
the $3$-distributive modular lattice
$M_n(p_1,p_2)$ is obtained by gluing quotients  $1/h$ with $a/0$
where $h$ is a coataom of $L_1$ and $a/0$  an atom of $S_n$ 
and $1/c$ with $p/0$ where $c$ is a coatom of $S_n$  and 
$p$ an atom of $L_2$.

 \begin{thm} (Freese, Herrmann, Huhn)
Assume  $p_1\# p_2$. Then no $M_n(p_1,p_2)$ embeds into
the congruence lattice of an algebra from a congruence modular
variety. On the other hand, for each $m$ there is $k$
such that for all $n>k$
and any $m$-generated sublattice of $M_n(p_1,p_2)$ there is an embedding 
into $L(V_1) \times L(V_2)$
where the $V_i$ are finite dimensional vector
spaces over the prime fields of characteristic $p_i$.
\end{thm}

The lattices in (4) arise from  a more general (and elaborate)
 gluing construction. Also, the proof is based
on the higher Arguesian identites valid in   lattices
of permuting equivalences.

\begin{thm} (Haiman)
For each field $F$ there is a sequence $A_n$ of modular lattices
not having any representation as lattice of permuting equivalences
and such that for each $m$ there is $k$ such that  for each $n\geq k$ 
and each  $m$-generated sublattice of $A_n$ there is an embedding into
the lattice $L(V_F)$  of subspaces of some finite dimensional $F$-vector space
$V_F$.
\end{thm}

\section{Non-finite axiomatizability}

The above (counter)examples yield families of  non-reprsentables
such that an ultraproduct is representable (of given type). 
Which implies that the class of representables
cannot be finitely axiomatized. 
This kind of appraoch to non-finite axiomatizablity 
 is due to Kirby Baker \cite{kirby} e.g. in the
case of the variety  generated by  modular lattices of width
at most $n$, where $n \geq 5$.

\begin{fact}\lab{8}
Every algebraic structure belongs to the universal class generated by its
finitely generated substructures.
Consequently, one has. 
\begin{itemize}
\item[(i)] If, for $p_1 \#p_2$,  $M$ is a non-principal ultraproduct
of the $M_n(p_1,p_2)$ then every finitely generated 
sublattice of $M$ is in the universal class
generated by $3$-distibutive sublattices of the lattices $L(V_i)\times L(V_2)$ 
where the $V_i$ range over finite dimensional vector spaces over
ptime fields of   characteristics $p_1, p_2$. See \cite{f0}
\item[(ii)]
If $A=\prod_{\mc F} A_n$ 
is a non-pricipal ultraproduct of the $A_n$ over $F$ then every finitely
generated sublattice of
$A$ is in the universal class generated 
by the lattices $L(V_F)$, $V_F$ ranging over finite dimensional $F$-verctor spaces. See \cite{hai1,hai2}
\end{itemize}
\end{fact}

\begin{cor}
Let  $\mc{L}$ a class of modular lattices closed
under formation of sublattices and satisfying one of the following:
\begin{itemize}
\item[(i)] 
Each $L\in \mc{L}$ 
embeds into the congruence lattice of same algebra in a congruence
modular variety and $\mc{L}$ contains, for given
characteristics $p_1\# p_2$ 
all $3$-distributive sublattices of $L(V_1)\times L(V_2)$ 
with $V_i$ ranging over finite dimensional vector
space over the prime field of characteristic $p_i$. 
\item[(ii)]
Each $L\in \mc{L}$ 
has a permuting equivalence representation
and $\mc{L}$ 
 contains for some field $F$, all lattices of subspaces of
finite dimensional $F$-vector spaces. 
\end{itemize}
 Then $\mc{L}$  
is not finitely axiomatizable. See \cite{f0,f1,f2,hai1,hai2}.
\end{cor}

\begin{proof}
 Assume $\mc{L}$ finitely axiomatized
whence, in particular, a universal class.
Then its complement 
$\mc{C }$ is an axiomatic class. 
Also, $L_n\in \mc{C}$ where  $L_n=M_n(p_1,p_2)$  resp..   $L_n=A_n$.
Let $L$ a non-principal ultraproduct of the $L_n$. Then also
$L \in \mc{C}$ since $\mc{C}$ is axiomatic.
But, in view of Fact~\ref{8}, $L \in \mc{L}$, a contradiction. 
\end{proof}

\end{document}